\documentclass[11pt]{amsart}

\usepackage{amsmath}
\usepackage{amssymb}
\usepackage{enumitem}
\usepackage{latexsym}
\usepackage{mathrsfs}
\usepackage{amsthm}
\usepackage[dvipsnames]{xcolor}
\usepackage[unicode,colorlinks=true,linktocpage=true,citecolor=Emerald,linkcolor=Fuchsia]{hyperref}
\usepackage[capitalise]{cleveref}

\numberwithin{equation}{section}
\newtheorem{theorem}{\bf Theorem}[section]

\newtheorem{lemma}[theorem]{\bf Lemma}
\newtheorem{proposition}[theorem]{\bf Proposition}

\theoremstyle{definition}

\newtheorem{definition}[theorem]{\bf Definition}

\numberwithin{equation}{section}
\makeatletter
\let\c@theorem\c@equation
\makeatother

\newtheorem*{namedtheorem}{\theoremname}
\newcommand{\theoremname}{testing}

\newcommand{\gen}[1]{\langle #1 \rangle}
\newcommand{\widebar}[1]
      {\overset{{\mskip6mu\leaders\hrule height0.4pt\hfill\mskip3mu}}{#1}
      \vphantom{#1}}
\newcommand{\ol}{\widebar}

\newcommand{\CC}{\mathbb{C}}
\newcommand{\FF}{\mathbb{F}}

\newcommand{\ZZ}{\mathbb{Z}}

\newcommand{\B}{\mathcal{B}}
\newcommand{\C}{\mathcal{C}}

\newcommand{\F}{\mathcal{F}}

\renewcommand{\L}{\mathcal{L}}

\renewcommand{\O}{\mathcal{O}}

\newcommand{\Q}{\mathcal{Q}}

\newcommand{\X}{\mathcal{X}}

\newcommand{\Z}{\mathcal{Z}}

\renewcommand{\phi}{\varphi}

\newcommand{\Hom}{\operatorname{Hom}}

\newcommand{\Aut}{\operatorname{Aut}}
\newcommand{\Out}{\operatorname{Out}}
\newcommand{\Inn}{\operatorname{Inn}}
\newcommand{\op}{\operatorname{op}}
\newcommand{\xra}{\xrightarrow}

\renewcommand{\mod}{\text{-}\textsf{mod}}

\newcommand{\bydef}{\overset{\mathrm{def}}{=}}
\newcommand{\Mack}{\operatorname{Mack}}
\renewcommand{\t}{\mathrm{t}}
\renewcommand{\r}{\mathrm{r}}

\newcommand{\tr}{\mathrm{tr}}
\newcommand{\iso}{\mathrm{iso}}

\newcommand{\hocolim}{\operatorname{hocolim}}

\newcommand{\Top}{\mathsf{Top}}
\newcommand{\hoTop}{\mathsf{hoTop}}

\newcommand{\longright}[1]{\;{\count255=0 \loop \relbar\mathrel{\mkern-6mu}%
    \advance\count255 by1\ifnum\count255<#1\repeat\rightarrow}\;}

\begin{document}
\title[Cohomology on the centric orbit category]{Cohomology on the centric orbit category of a fusion system}
\author{George Glauberman}
\address{Department of Mathematics \\ University of Chicago \\ 5734 S.
University Ave \\ Chicago, IL 60637}
\email{gglauber@uchicago.edu}
\author{Justin Lynd}
\address{Department of Mathematics \\ University of Louisiana at Lafayette \\ Lafayette, LA 70504}
\email{lynd@louisiana.edu}
\date{\today}
\thanks{The first author was partially supported by a Simons Foundation
Collaboration Grant. The second author was partially supported by NSF Grant
DMS-1902152. The authors thank these organizations for their support. The
authors would like to thank the Isaac Newton Institute for Mathematical
Sciences, Cambridge, for support and hospitality during the program Groups,
representations and applications, where work on this paper was undertaken and
supported by EPSRC grant no EP/R014604/1.}

\begin{abstract}
We study here the higher derived limits of mod $p$ cohomology on the centric
orbit category of a saturated fusion system on a finite $p$-group. It is an
open problem whether all such higher limits vanish.  This is known in many
cases, including for fusion systems realized by a finite group and for many
classes of fusion systems which are not so realized.  We prove that the higher
limits of $H^j$ vanish provided $j \leq p-2$, by showing that
the same is true for the contravariant part of a simple Mackey composition
factor of $H^j$ under the same conditions. 
\end{abstract}

\date{\today}

\maketitle

\section{Introduction} 

The purpose of this note is to show that $H^j(-,\FF_p)$, considered as a
contravariant functor on the centric orbit category of a saturated fusion
system on a $p$-group, has vanishing higher derived limits provided $j \leq p-2$.

\begin{theorem}
\label{T:main}
Fix a prime $p$ and a nonnegative integer $j \leq p-2$.
For each saturated fusion system $\F$ on a finite $p$-group,
\[
{\lim}^i H^j(-, \FF_p)|_{\O(\F^c)} = 0 
\]
for all $i \geq 1$. 
\end{theorem}

Here $\F^c$ denotes the full subcategory of $\F$ with objects the $\F$-centric
subgroups, and $\O(\F^c)$ denotes the corresponding orbit category.  For
example, the special case $j = 1$ of Theorem~\ref{T:main} states that the
functor $H^1(-,\FF_p)\colon \O(\F^c)^{\op} \to \FF_p\mod$ is acyclic when $p$ is odd.
This was a case that motivated our work for reasons we outline below.  But
first we give context for this result, referring to
\cite[Section~2]{BrotoLeviOliver2003} and
\cite[Section~III.5.6]{AschbacherKessarOliver2011} for further details. 

One way to study the classifying space of a finite group $G$ at a prime $p$ (or
of a compact Lie group, or of a saturated fusion system over a $p$-group) is to
recognize the space as glued from classifying spaces of collections of proper
subgroups. Following Dwyer's uniform approach to such homotopy decompositions
in the 1990s \cite{Dwyer1997}, this comes in the form of a map
\[
\hocolim_\C F \to BG
\]
inducing an isomorphism in mod $p$ cohomology, where $\C$ is a small category
and $F \colon \C \to \Top$ is some functor such that for each $c \in \C$, the
composite $F(c) \to BG$ identifies $F(c)$ with the classifying space of a
subgroup of $G$ up to homotopy. The decomposition most relevant here is the
\emph{subgroup decomposition for the centric collection} \cite{Dwyer1997}.  In
this case, $\C = \O_p^c(G)$ is the full subcategory of the orbit category of
$G$ on the $p$-centric subgroups, and $F(P) = \widetilde{B}P$ is an appropriate
lifting to $\Top$ of the classifying space functor $B\colon \O_{p}^c(G) \to
\hoTop$ to the homotopy category. For example, one can use $\widetilde{B}P = EG
\times_G G/P$, the Borel construction applied to the orbit $G/P$ considered as
a $G$-space \cite[Section 5.2]{BensonSmith2008}. 

On the other hand, if $\F$ is a saturated fusion system on a $p$-group
$S$, there may be no finite group with Sylow $p$-subgroup $S$
realizing the fusion in $\F$ (i.e., $\F$ may be exotic).  Consequently there is
no longer an obvious lifting of the classifying space functor. As was noticed
by Broto, Levi, and Oliver \cite[\S 2]{BrotoLeviOliver2003}, the Dwyer-Kan
obstructions to such a lifting \cite{DwyerKan1992} are the same as the
obstructions to the existence and uniqueness of a centric linking system $\L$
associated with $\F$. By a theorem of Chermak, these obstructions vanish
\cite{Chermak2013}, and the $p$-completion $|\L|_p^\wedge$ is then regarded as
a classifying space ``$B\F$'' for the fusion system. The subgroup decomposition
in this context takes the form of a mod $p$ cohomology isomorphism
$\hocolim_{\O(\F^c)} \widetilde{B} \to |\L|$, where $\widetilde{B}$ is the left
homotopy Kan extension along the quotient functor $\tilde{\pi} \colon \L \to
\O(\F^c)$ of the constant functor $\L \to *$
\cite[Proposition~III.5.29]{AschbacherKessarOliver2011}. 

The most immediate application of the subgroup decomposition for a saturated
fusion system is to the computation of cohomology of the linking system. The
mod $p$ cohomology $H^{*}(|\L|,\FF_p)$ of the linking system is the abutment
of the Bousfield-Kan spectral sequence for the homotopy colimit with $E_2$
page given by $E_2^{i,j} = {\lim}^i_{\O(\F^c)} H^j(-,\FF_p)$. Note that by a
general property of $p$-local functors on orbit categories, this page is
bounded to the right \cite[Corollary~3.4]{BrotoLeviOliver2003}.
It is an open problem whether the functor $H^j(-,\FF_p)$ is acyclic,
i.e., whether ${\lim}^i_{\O(\F^c)} H^j(-,\FF_p) = 0$ for all $i \geq 1$. For
example, see \cite[Problem~7.12]{AschbacherOliver2016} and
\cite[Conjecture]{DiazPark2015}.  Dwyer called a homology decomposition with
this property ``sharp''. Sharpness would immediately yield the most natural
application of the subgroup decomposition, which is that the cohomology of the
linking system satisfies a Cartan-Eilenberg stable elements formula: 
\[
H^{j}(|\L|,\FF_p) \cong \lim_{P \in \O(\F^c)} H^j(P,\FF_p).  
\] 
The stable elements formula was already shown by Broto, Levi, and Oliver
\cite[Theorem~5.8]{BrotoLeviOliver2003} (along with other important
consequences for combinatorial descriptions mapping spaces), but the proof
relies on difficult results from homotopy theory. It would be very nice to see
the stable elements formula directly via sharpness of the subgroup
decomposition.

Sharpness over the centric $p$-orbit category of a finite group $G$ was shown
by Dwyer \cite[\S 10]{Dwyer1998}, and over the centric orbit category of
$\F_S(G)$ by D{\'i}az and Park \cite[Theorem~B]{DiazPark2015}.  Sharpness has
since been established for certain families of exotic fusion systems, notably:
those on $p$-groups with an abelian subgroup of index $p$ \cite{DiazPark2015},
the smallest Benson-Solomon fusion system \cite{HenkeLibmanLynd2023}, all
fusion systems of characteristic $p$-type/local characteristic $p$ (in the
sense of the CFSG) \cite{HenkeLibmanLynd2023}, and for the 27 exotic fusion
systems on the Sylow $p$-subgroup of $G_2(p)$ \cite{GrazianMarmo}.  The
sharpness problem has been studied from a very general point of view also by
Yal\c{c}{\i}n \cite{Yalcin2022}, who showed that sharpness for the subgroup and
the normalizer decompositions are equivalent. 

We follow D{\'i}az and Park \cite{DiazPark2015} by regarding cohomology as a Mackey
functor for the fusion system, and we study the simple Mackey functors
$S_{T,V}$ occurring as composition factors of $H^j$. Theorem~\ref{T:main} is
ultimately deduced from the following stronger result. We write $k = \FF_p$ for
short. 

\begin{theorem}
\label{STVacyc}
Fix a saturated fusion system $\F$ on a $p$-group $S$, a subgroup $T \leq S$,
and a simple $k\Out_\F(T)$-module $V$. If the simple Mackey functor $S_{T,V}$
is a composition factor of $H^j(-,k)$ and $j \leq p-2$, then the restriction
${S_{T,V}}^*|_{\O(\F^c)}$ of the contravariant part of $S_{T,V}$ has vanishing
higher derived limits.  
\end{theorem}

Theorem~\ref{STVacyc} is amenable to the standard technique of ``pruning'',
where one filters ${S_{T,V}}^*$ by subquotient functors (not themselves
contravariant parts of Mackey functors) that take the value $0$ except on a
single $\F$-conjugacy class of subgroups of $S$.  Thus, our proof of
Theorem~\ref{T:main} goes by filtering $H^j$ first as a Mackey functor
completely, and then second as a coefficient system.

One reason for our looking at this problem grew out of a loose analogy with the
paper \cite{GlaubermanLynd2016}, where we studied higher limits of the center
functor $\Z_\F\colon \O(\F^c)^{\op} \to \ZZ_{(p)}$\mod, $P \mapsto Z(P)$ in the
context of Oliver's proof \cite{Oliver2013} of Chermak's Theorem
\cite{Chermak2013} on centric linking systems. That proof proceeds by a
reduction to the case where $\F$ is realizable by a finite $p$-constrained
group $\Gamma$ with normal centric $p$-subgroup $Q$. The observation of
\cite{GlaubermanLynd2016} was the relevance of finding a $p$-local subgroup $H$
of $\Gamma$ that controls fixed points on $Z(Q)$, i.e.  that satisfies
$C_{Z(Q)}(H) = C_{Z(Q)}(\Gamma)$, a problem that had been studied by the first
author under the guise of ``control of weak closure of elements''.  This
motivated us to look at whether techniques for ``controlling transfer'' in
finite groups could be useful in studying the higher limits of the functor
$H^1(-,\CC^{\times})$ and its subfunctor $H^1(-,\FF_p)$. The issue is that it
appears difficult to get reductions similar to those for the center functor in
order for these techniques to be applicable. Also, general techniques for
controlling transfer are known only when $p \geq 5$, so in the end what
Theorem~\ref{T:main} gives in the case $j = 1$ is stronger than what those
methods seemingly would have yielded even if they had been applicable.

We would like to thank Antonio D{\'i}az for corrections and helpful suggestions
on a previous version of this article.

\section{Background results}

\subsection{Nilpotent action on group cohomology}

Let $k$ be a commutative ring with identity.  If $G$ is a finite group, $H$ is
a subgroup of $G$, and $M$ is a $kG$-module, then we use the usual notation
$\tr_H^G\colon M^H \to M^G$ for fixed points and the relative trace map. If $p$
is a rational prime which is zero in $k$, then the relative trace is zero in
cases where there is an element of $G$ outside $H$ but normalizing $H$ and
acting with small nilpotence degree on $M$. We state this when $k = \FF_p$, the
only case we need. 
\begin{lemma}
\label{L:tracebasic}
Let $G$ be a finite group, $p$ a prime, and $V$ an $\FF_p[G]$-module.  Suppose
$g$ is an element of $G$ of $p$-power order such that $(g-1)^{p-1}V = 0$.  
Then $\tr_H^G(V) = 0$ for every subgroup $H$ of $G$ with $g \in N_G(H)-H$. 
\end{lemma}
\begin{proof}
Decompose $\tr_H^G =
\tr_{H\gen{g}}^G\tr_{H\gen{g^p}}^{H\gen{g}}\tr_{H}^{H\gen{g^p}}$. For each $v
\in \tr_H^{H\gen{g^p}}(C_V(H))$, we have 
\[
\tr_{H\gen{g^p}}^{H\gen{g}}(v) = (1+g+\cdots + g^{p-1})v = (g-1)^{p-1}v = 0. 
\]
\end{proof}

In this paper, if $X$ and $Y$ are two subsets of a group $G$, we write $[X,Y]$ for the subgroup of $G$ generated by the set of commutators $[x,y] = xyx^{-1}y^{-1}$ with $x \in X$ and $y \in Y$.
Our iterated commutators are right-associated: set $[X,Y; 1] = [X,Y]$, and inductively $[X,Y; i] = [X, [X,Y; i-1]]$ for $i \geq 2$.
If $G$ has a left action on some module $V$, the notation $[X,V; i]$ should be interpreted in the semidirect product of $V$ by $G$, in which case $[X,V;i]$ is a subspace of $V$, and $[x,v; i] = (x-1)^iv$ for all $x \in G$ and $v \in V$. 

The techniques we have generally take advantage of situations in a finite
$p$-group in which some subgroup $G$ normalizes another subgroup $P$ and acts
with small nilpotence degree on it, usually action which is quadratic (or trivial):
$[G,G,P] =  1$. 
Then the following lemma of Miyamoto \cite[Lemma~2]{Miyamoto1981} provides a
bound on the nilpotence degree of the action of $G$ on $H^j(P,A)$ that is linear in $j$ when 
$A$ is finite abelian. 

\begin{lemma}
\label{L:miyamoto}
Let $P$ be a finite $p$-group, $A$ a finite abelian group with trivial
$P$-action, and $G$ a finite group acting on $P$ and $A$.  Assume
$h$ and $n$ are nonnegative integers such that $(g-1)^h$ acts as zero on each
$G$-composition factor of $P$, and such that $(g-1)^n$ acts as zero on each
$G$-composition factor of $A$. Then for all $j \geq 0$, $(g-1)^{(h-1)j+n}$ acts
as zero on each $G$-composition factor of $H^j(P,A)$. 
\end{lemma}

\subsection{Mackey functors for fusion systems}

The notation we use for fusion systems follows
\cite{AschbacherKessarOliver2011}. We apply morphisms from right to left. Let
$\F$ be a saturated fusion system on a finite $p$-group $S$. A subgroup $Q$ of
$S$ is \emph{fully $\F$-normalized} (respectively, \emph{fully
$\F$-centralized}) if $|N_S(Q)| \geq |N_S(Q')|$ (respectively, $|C_S(Q)| \geq
|C_S(Q')|$) for each conjugate $Q'$ of $Q$ in $\F$, i.e., for each subgroup of
the form $\phi(Q)$ with $\phi \in \Hom_\F(Q,S)$.  A subgroup $Q$ of $S$ is
$\F$-\emph{centric} if $C_S(Q') \leq Q'$, i.e. $C_S(Q') = Z(Q')$, for each
$\F$-conjugate $Q'$ of $Q$.  The symbol $\F^c$ denotes the set of $\F$-centric
subgroups, and also the full subcategory with the same objects. Note that a
subgroup of $S$ is $\F$-centric if and only if it is fully centralized and
contains its centralizer in $S$
\cite[Definition~I.3.1]{AschbacherKessarOliver2011}. By one of the axioms for
saturation, a fully normalized subgroup $P \leq S$ is also fully centralized
and fully automized: $\Aut_S(P)$ is a Sylow $p$-subgroup of $\Aut_\F(P)$
\cite[Proposition~I.2.5]{AschbacherKessarOliver2011}. 

For each pair of subgroups $P,Q \leq S$, $\Inn(Q)$ acts on $\Hom_\F(P,Q)$ by
left composition.  The orbit category $\O(\F)$ of $\F$ is the category with the
same objects as $\F$ and with morphism sets 
\[ 
\Hom_{\O(\F)}(P,Q) = \Inn(Q)\backslash\Hom_\F(P,Q),
\]
the orbits under this action. If $\X$ is a collection of subgroups of $S$ which
is closed under $\F$-conjugacy and also closed under passing to overgroups in
$S$, then we abuse notation by using $\X$ also for the full
subcategory of $\F$ with object set $\X$, and we write $\O(\X)$
for the corresponding orbit category.  Other than the full orbit
category itself, we will only need to work with the centric orbit category,
the case $\X = \F^c$. 

Since the morphisms in a fusion system model conjugation of $p$-subgroups in a
finite group, it is natural that there is a notion of Mackey functor for fusion
systems. We first want to recall from \cite[Section~2]{DiazPark2015} the
definition of a Mackey functor in this setting in the form that is most useful
later. Let $k$ be a commutative ring with identity.  Let $M = (M^*,M_*)$ be a
pair of functors from $\O(\F)$ to $k\mod$ with $M^*$ contravariant and $M_*$ covariant.  
Set $\r_{P}^Q = M^*([\iota_P^Q])$, $\t_P^Q = M_*([\iota_P^Q])$, and
$\iso([\phi]) = M_*([\phi])$ for each $P \leq Q \leq S$ and each isomorphism
$[\phi] \colon P \to \phi(P)$ in $\O(\F)$. Then $M$ is a Mackey functor for
$\F$ if the following conditions hold \cite[Definition~2.1,
Proposition~2.2]{DiazPark2015}. 
\begin{enumerate}
\item 
$M(P) \bydef M^*(P) = M_*(P)$ for each $P \leq S$,
\item (Isomorphism) $M_*([\phi]) \bydef \iso(\phi) =  M^*([\phi]^{-1})$ for
each isomorphism $[\phi]$ in $\O(\F)$, and
\item (Mackey formula) for each $P,Q \leq R \leq S$, 
\[
r_Q^R \circ t_P^R = \sum_{x \in [Q\backslash R/P]} t_{ Q \cap { }^xP}^Q \circ r_{Q \cap { }^xP}^{{ }^xP} \circ \iso(c_x|_P). 
\]
\end{enumerate}

A morphism $M \to N$ of Mackey functors is a family of $k$-module
homomorphisms $\eta_P\colon M(P) \to N(P)$ such that $\eta = (\eta_P)$ is both
a natural transformation from $M^* \to N^*$ and a natural transformation $M_*
\to N_*$ simultaneously. A subfunctor of $M$ is a subfunctor of $M^*$ which is
simultaneously a subfunctor of $M_*$, and quotient
functors are defined objectwise. 

\subsection{Simple Mackey functors}\label{simple mackey}

The simple objects in $\Mack_k(\F)$ are parametrized by pairs $(T,V)$, where
$T$ is a subgroup of $S$ taken up to $\F$-conjugacy, and where $V$ is a simple
(irreducible) $k\Out_\F(T)$-module taken up to isomorphism
\cite[Section~3]{DiazPark2015}.  When convenient we view $V$ as a
$k\Aut_\F(T)$-module via inflation.  The corresponding simple Mackey functor
$S_{T,V}$ has the property that $S_{T,V}(T) = V$ and $S_{T,V}(Q) = 0$ for all
subgroups $Q$ such that $T$ is not $\F$-conjugate to a subgroup of $Q$. 

Let $T \leq S$ and let $V$ be a simple $k\Out_\F(T)$-module. We give the
description of the functor $S_{T,V}$ on objects and isomorphisms in $\O(\F)$
from p.153 of \cite{DiazPark2015}, since this will be important for the proof
of Theorem~\ref{T:main}, but interestingly the effect of $S_{T,V}$ on
nonisomorphisms is not so important for our argument.  For that we refer the
interested reader to the description in \cite{DiazPark2015}.  Our treatment is
a little different from (but equivalent to) that in \cite{DiazPark2015}, since
we need to pay somewhat closer attention to precisely how $S_{T,V}(Q)$
decomposes as a direct sum of $k\Out_\F(Q)$-modules. 

The set $\Hom_\F(T,Q)$ is an $\Aut_\F(Q)$-$\Aut_\F(T)$ biset with action on
either side given by composition. The orbits $\Hom_\F(T,Q)/\Aut_\F(T)$ are in
correspondence with the set of subgroups of $Q$ that are $\F$-conjugate to $T$.
Likewise the double orbits $\Aut_\F(Q)\backslash \Hom_\F(T,Q)/\Aut_\F(T)$ are
in correspondence with the $\Aut_\F(Q)$-orbits of such subgroups.  Let 
\[
A_{T,Q} = [\Aut_\F(Q)\backslash \Hom_\F(T,Q)/\Aut_\F(T)]
\]
be a set of representatives for the double orbits.

For each $\alpha \in \Hom_\F(T,Q)$, we 
temporarily set $U = \alpha(T)$ and
denote (formally) by $\alpha \otimes V$
the $k\Aut_\F(U)$-module, isomorphic to $V$ as a $k$-module via $\alpha
\otimes v \mapsto v$, with action
\[
\phi \cdot (\alpha \otimes v) = \alpha \otimes \alpha^{-1}\phi\alpha v. 
\]
for each $\phi \in \Aut_\F(U)$ and $v \in V$.  Since $T$ acts trivially
on $V$, $U$ acts trivially on $\alpha \otimes V$. Set 
\begin{equation}
\label{Walpha}
W_{\alpha} = \tr_{U}^{N_Q(U)}(\alpha \otimes V),
\end{equation}
the image of the relative trace, 
where here
$N_Q(U)$ acts on the $\Out_\F(U)$-module $\alpha \otimes V$ through the composite
\[
N_Q(U) \twoheadrightarrow N_{\Inn(Q)}(U) \twoheadrightarrow \Aut_{N_Q(U)}(U) \to \Aut_\F(U)
\]
with $UC_Q(U)$ acting trivially.
Note $W_\alpha$ is a $k$-submodule of $\alpha \otimes V$. 
It has the structure of a $kN_{\Aut_\F(Q)}(U)$-module on which $N_Q(U)$ acts trivially. 

Write $U^Q$ for the $Q$-conjugacy class of $U$, and
$N_{\Aut_\F(Q)}(U^Q)$ for the stabilizer of this class in $\Aut_\F(Q)$.
Since $\Inn(Q)$ is a normal subgroup of $\Aut_\F(Q)$ that acts transitively on $U^Q$,
\[
N_{\Aut_\F(Q)}(U^Q) = N_{\Aut_\F(Q)}(U)\Inn(Q).
\]
By construction, the subgroup $N_{\Aut_\F(Q)}(U) \cap \Inn(Q) = N_{\Inn(Q)}(U)$ acts trivially on $W_\alpha$, 
and we may regard $W_\alpha$ as a module for $N_{\Aut_{\F}(Q)}(U^Q)$ via the composite
\[
N_{\Aut_\F(Q)}(U^Q) \twoheadrightarrow N_{\Aut_\F(Q)}(U^Q)/\Inn(Q) \cong N_{\Aut_\F(Q)}(U)/N_{\Inn(Q)}(U). 
\]

Set now
\begin{equation}
\label{STVQalpha}
S_{T,V}(Q)_{\alpha} =
W_\alpha\!\uparrow_{N_{\Aut_\F(Q)}(U^Q)}^{\Aut_\F(Q)} =
\bigoplus_{\phi} W_{\phi\alpha},
\end{equation}
where $\phi$ runs over a set of representatives for the left cosets of $N_{\Aut_\F(Q)}(U^Q)$ in $\Aut_\F(Q)$. 
Then $S_{T,V}(Q)_\alpha$ is a $k\Out_\F(Q)$-module, that is, $Q$ still acts trivially. 
The value of $S_{T,V}$ on the subgroup $Q$ is then
\begin{equation}
\label{STVQ}
S_{T,V}(Q) = \bigoplus_{\alpha \in A_{T,Q}} S_{T,V}(Q)_{\alpha},
\end{equation}
an $\Out_\F(Q)$-invariant direct sum decomposition. 

Now let $Q'$ be another subgroup of $S$ and $\beta \colon Q \to Q'$ an
isomorphism in $\F$. The bijection $\Hom(T,Q) \xra{\beta \circ -}
\Hom(T,Q')$ determines the $k$-module isomorphism $\alpha \otimes V \cong
\beta\alpha \otimes V$ intertwining the actions of $N_{\Aut_\F(Q)}(\alpha(T))$
and $N_{\Aut_\F(Q')}(\beta(\alpha(T)))$ with respect to conjugation by $\beta$.
It induces an isomorphism of $k$-modules 
\begin{equation}
\label{Wbetaalpha}
W_\alpha \cong W_{\beta \circ \alpha}
\end{equation}
intertwining the actions of $N_{\Aut_\F(Q)}(\alpha(T)^Q)$ and
$N_{\Aut_\F(Q')}(\beta(\alpha(T))^{Q'})$.  The component of $\iso(\beta)$ at
$\alpha$ is the corresponding map of induced modules
\[
\iso(\beta)_\alpha \colon S_{T,V}(Q)_\alpha \to S_{T,V}(Q')_{\beta\alpha}. 
\]
and $\iso(\beta)$ is the sum of these maps. 

Note that if $V$ is a simple $\FF_p\Out_\F(T)$-module, then $S_{T,V}$ takes values in
$\FF_p$\mod. Further, since $V$ is simple, $S_{T,V}$ is a simple functor
\cite{DiazPark2015}. 

\subsection{Higher limits of functors on orbit categories}

Let $k$ be a commutative $\ZZ_{(p)}$-algebra and let $M$ be a contravariant
functor from the orbit category of a group or a fusion system to the category
of $k$-modules. A common technique for computing the higher limits of $M$ (and
especially for showing that such higher limits vanish) is to use a filtration
of $M$ each of whose successive quotient functors is atomic, namely a functor
which vanishes except on a single conjugacy class of subgroups. This method
does not always work as, for example, in the case of the center functor
\cite{Oliver2018}. But the idea is often effective for making reductions even
when it doesn't work directly. And ultimately it is all that is needed for the
proof of the main theorem here. 

Let $G$ be a finite group and let $A$ be a $\ZZ_{(p)}G$-module. Define a functor
\[
F_{A} \colon \O_p(G)^{\op} \to \ZZ_{(p)}\mod
\]
via $F_{A}(1) = A$ and $F_{A}(P) = 0$ when $P \neq 1$. 
The action of $G = \Aut_{\O_p(G)}(1)$ on $A = F_{A}(1)$ is the given one.
The higher limits of $F_{A}$ are denoted $\Lambda^i(G,A)$ and arise as the higher limits of atomic functors, as was first shown by Jackowski, McClure, and Oliver \cite[Proposition~5.20]{AschbacherKessarOliver2011}. 

\begin{proposition}\label{P:lim-atomic}
Let $F$ be any functor on the orbit category of a fusion system $\F$ which vanishes except on the $\F$-conjugacy class of a subgroup $P$.  
Then there is an isomorphism $\lim^* F \cong \Lambda^*(\Out_\F(P),F(P))$.
\end{proposition}

Let $P_1,\dots,P_n$ be a set of representatives for the $\F$-conjugacy classes
of subgroups of $S$ such that if $i < j$, then $P_j$ is not conjugate to a
subgroup of $P_i$. Then one can make a filtration $0 = M_0 \subset M_{1}
\subseteq \cdots \subset M_n = M$, in which $M_j$ is the functor equal to $M$
on the union of the conjugacy classes $P_i$ with $i \leq j$, and $0$ otherwise.
Then $(M_{i}/M_{i-1})(P) = M(P)$ if $P$ is conjugate to $P_i$, and it is zero
otherwise, i.e. the quotient is atomic.  

In general, we use the notation $M_Q$ for the atomic subquotient functor
of $M$ corresponding to the $\F$-conjugacy class of $Q$, namely the functor
with values $M_Q(P) = M(P)$ if $P$ is $\F$-conjugate to $Q$, and $M_Q(P) = 0$
otherwise.

The next lemma is proved using long exact sequences on higher limits
corresponding to short exact sequences of functors arising out of a filtration of the above type. 

\begin{lemma}[{\cite[Corollary~5.21(a)]{AschbacherKessarOliver2011}}]
\label{L:filterbyatomic}
Let $M$ be a contravariant functor on the centric orbit category of a fusion system $\F$. 
Assume that $M_Q$ is acyclic for all $Q \in \F^c$, i.e., $\Lambda^m(\Out_{\F}(Q),M(Q)) = 0$ for all $m \geq 1$.
Then $M$ is acyclic. 
\end{lemma}

The functors $\Lambda^*(G,M)$ vanish in many cases. 
See for example Section~III.5 of \cite{AschbacherKessarOliver2011} for many results along these lines.
In the next lemma we state two such vanishing results.  

Recall that a radical $p$-chain of length $m$ in the finite group $G$ is a
sequence $O_p(G) = P_0 < P_1 < \cdots P_m$ such that $P_i =
O_p(N_{G}(P_1,\dots,P_i))$ for each $i = 0, \dots, n$, where here
$N_G(P_1,\dots,P_i)$ denotes the intersection of the normalizers in $G$ of the
$P_i$. 

\begin{lemma}
\label{L:AKO5.27}
Let $G$ be a finite group, $M$ a $\ZZ_{(p)}G$-module, and $m \geq 1$.  
\begin{enumerate}[label=\textup{(\arabic*)}, ref=\arabic*]
\item \label{Lambda:O_p} If $O_p(G) \neq 1$, then $\Lambda^m(G,M) = 0$. 
\item \label{Lambda:radical} If $\tr_{1}^{N_G(P_1 ,\dots, P_m)}(M) = 0$ for
each radical $p$-chain $1 = P_0 < P_1 < \cdots < P_m$ of length $m$ in $G$,
then $\Lambda^m(G,M) = 0$.
\end{enumerate}
\end{lemma}
\begin{proof}
For \eqref{Lambda:O_p}, see
\cite[Proposition~III.5.24(b)]{AschbacherKessarOliver2011}. Then
\eqref{Lambda:radical} is a restatement of
\cite[Proposition~III.5.27]{AschbacherKessarOliver2011}, given
\eqref{Lambda:O_p}. 
\end{proof}

We want to show (\cref{STVacyc}) that the restriction of the contravariant part of each Mackey composition factor $S_{T,V}$ of $H^j(-,\FF_p)$ is acyclic when $j \leq p-2$.
For doing this, D{\'i}az and Park show we can restrict attention to composition factors $S_{T,V}$ with $T$ not $\F$-centric.

\begin{lemma}
\label{L:DP43} 
Let $k$ be a field of characteristic $p$ and $\F$ a saturated fusion system on the finite $p$-group $S$. 
For each $\F$-centric subgroup $T$ and each simple $k\Out_\F(T)$-module $V$, the restriction 
${S_{T,V}}^*|_{\O(\F^c)}$ of the contravariant part of $S_{T,V}$ is acyclic. 
\end{lemma}
\begin{proof}
Proposition~3.3 of \cite{DiazPark2015} implies that when $T$ is centric, ${S_{T,V}}^*|_{\O(\F^c)}$ is an $\F^c$-restricted
Mackey functor for $\F$ in the sense of Definition~2.1 of \cite{DiazPark2015}. 
The lemma then follows from Theorem~A there. 
\end{proof}

\section{Proof of Theorem~\ref{T:main}}

Throughout this section we fix a prime $p$,
and a saturated fusion system $\F$ on the $p$-group $S$. 
We set $k = \FF_p$ for short. 
For fixed $j \geq 0$, we consider $H^j(-,k)$ as a Mackey functor on $\O(\F)$
where the contravariant structure is induced by restrictions and conjugations,
and where the covariant structure is induced by transfers and conjugations (as
usual). 

We first fix some additional notation that we keep for the remainder of the section. 

Let $\B$ be the collection of all normal subgroups $B$ of $S$ such that
\begin{eqnarray}
\label{E:B}
C_S(B) \leq B \quad \text{ and } \quad [B,B,S] \overset{\mathrm{def}}{=} [B,[B,S]] =  1. 
\end{eqnarray}
By \cite[5.3.12]{Gorenstein1980}, each subgroup $B \leq S$ maximal subject to
being normal and abelian coincides with its centralizer in $S$.  Since $[B,S]
\leq B$, we have $[B,B,S] \leq [B,B] = 1$. Thus, $\B$ is nonempty.
Further, since each member of $\B$ is normal in $S$, it is fully
normalized, hence fully centralized by one of the saturation axioms for $\F$.
This implies $\B \subseteq \F^c$.

\begin{definition}
\label{scriptQ}
Let $T$ be a subgroup of $S$. Define $\Q$ to be the set of pairs $(Q,\alpha)$
consisting of a centric subgroup $Q \in \F^c$ and a morphism $\alpha \in
\Hom_\F(T,Q)$ having the property that there are $B \in \B$ and an isomorphism
$\beta \colon Q \to Q'$ in $\F$ such that $B \cap \beta\alpha(T) < B \cap
\beta(Q)$.
\end{definition}

The proof of Theorem~\ref{T:main} is broken into two propositions. In the first
one, we show that the $k\Out_\F(Q)$-submodule $S_{T,V}(Q)_{\alpha}$ of
$S_{T,V}(Q)$ (see equation \eqref{STVQalpha}) is 0 whenever $S_{T,V}$ is a composition
factor of $H^j(-,\FF_p)$, $(Q,\alpha) \in \Q$, and $j \leq p-2$. In the second one,
we use this to show that the atomic subquotient $({S_{T,V}}^*)_Q$ is acyclic for an
arbitrary $\F$-centric subgroup $Q$ when $T$ is not $\F$-centric. 

\begin{proposition}
\label{P:MP=0} 
Fix a prime $p$, a saturated fusion system $\F$ on a finite $p$-group $S$, a
nonnegative integer $j$, a subgroup $T \leq S$, and a simple
$k\Out_\F(T)$-module $V$.
If $j \leq p-2$ and $S_{T,V}$ is a composition factor of $H^j(-,k)$, 
then $S_{T,V}(Q)_\alpha = 0$ for all $(Q,\alpha) \in \Q$.  
\end{proposition}
\begin{proof}
Set $U = \alpha(T)$ and $W_\alpha =
\tr_U^{N_Q(U)} (\alpha \otimes V)$.  
By \eqref{STVQalpha}, $S_{T,V}(Q)_\alpha$ takes the form
\[
S_{T,V}(Q)_\alpha = W_\alpha\!\uparrow_{N_{\Aut_\F(Q)}(U^Q)}^{\Aut_\F(Q)}. 
\]
That is $S_{T,V}(Q)_{\alpha}$ is induced from $W_\alpha$. 

Let $(Q,\alpha) \in \Q$. 
By \cref{scriptQ} there is an $\F$-isomorphism $\beta\colon Q \to Q'$ and
$B \in \B$ such that for $U' = \beta(U)$, we have $B \cap U' < B \cap Q'$.  The
map $\beta$ induces an intertwining $W_\alpha \cong W_{\beta \circ \alpha}$.
Because of this we may as well change to lighter notation by replacing $Q$ by
$Q'$, $U$ by $U'$, and $\alpha$ by $\beta \circ \alpha$. Thus, $B \cap U < B
\cap Q$ and we want to show $W_\alpha = 0$.

Set $B_Q = B \cap Q$ for short. 
Since $B$ is normal in $S$, $B_Q$ is normal in $Q$ and $N_{B_Q}(U)$ is normal in $N_Q(U)$.
Also, as $B \cap U < B_Q$, we have $U < U B_Q$.  
Hence $U < N_{UB_Q}(U) = UN_{B_Q}(U)$. 
Since $S_{T,V}$ is a composition factor of $H^j(-,k)$, 
    $V$ is a $k\Out_\F(T)$-composition factor of $H^j(T,k)$, 
    and $\alpha \otimes V$ is an $\Out_\F(U)$-composition factor of $H^j(U,k)$. 
We have $[N_{B_Q}(U), N_{B_Q}(U), U] \leq [B,B,S] = 1$; 
in particular $N_{B_Q}(U)$ acts quadratically 
    on every $\Aut_\F(U)$-composition factor of $U$.
The hypotheses of \cref{L:miyamoto} thus hold with $h = 2$ and $n = 1$, and with $U$, $k$, and $\Aut_\F(U)$
    in the roles of $P$, $A$, and $G$. 
As $\alpha \otimes V$ is a composition factor of $H^j(U,k)$, by that lemma we have 
$(b-1)^{j+1}(\alpha \otimes V) = 0$ 
for all $b \in N_{B_Q}(U)$.
Since composition factors are always $\FF_p$-vector spaces and $j+1 < p$, \cref{L:tracebasic} implies $\tr_U^{UN_{B_Q}(U)}(\alpha \otimes V) = 0$.
Hence
\[ 
W_\alpha = \tr_{U}^{N_{Q}(U)}(\alpha \otimes V) 
         = \tr_{UN_{B_Q}(U)}^{N_{Q}(U)}( \tr_{U}^{UN_{B_Q}(U)}(\alpha \otimes V))
= 0, 
\] 
and this completes the proof.
\end{proof}

\begin{proposition}
\label{P:pruning} 
Fix a prime $p$, a saturated fusion system $\F$ on a finite $p$-group $S$,
and a nonnegative integer $j$.  
Let $S_{T,V}$ be a composition factor of $H^j(-,k)$ such that $T$ is not $\F$-centric. 
If $j \leq p-2$, then for all $Q \in \F^c$, the atomic subquotient functor $({S_{T,V}}^*)_Q$ is acyclic. 
\end{proposition} 
\begin{proof}
Fix $Q \in \F^c$. The functor $({S_{T,V}}^*)_Q$ does not depend on $Q$,
but only on the $\F$-conjugacy class of $Q$. By Proposition~\ref{P:lim-atomic},
we have
\[
{\lim}^*_{\O(\F^c)} ({S_{T,V}}^*)_Q \cong \Lambda^*(\Out_\F(Q'), S_{T,V}(Q')),
\]
for any $\F$-conjugate $Q'$ of $Q$. So we may assume $Q$ to be fully normalized
in $\F$. In particular, $R := \Out_S(Q)$ is a Sylow $p$-subgroup of $G :=
\Out_\F(Q)$.  By \cref{L:AKO5.27}\eqref{Lambda:O_p}, the proposition holds if
$O_p(G) \neq 1$, so we are reduced to $O_{p}(G) = 1$.

Adopt the notation of \cref{simple mackey}.  By \eqref{STVQ} and additivity of
the functors $\Lambda^*(G,-)$, we have
\begin{equation}
\label{Lambdadecomp}
\Lambda^*(G,S_{T,V}(Q)) = \bigoplus_{\alpha \in A_{T,Q}} \Lambda^*(G,S_{T,V}(Q)_\alpha). 
\end{equation}
Fix arbitrary $B \in \B$ and $\alpha \in A_{T,Q}$. Set $B_Q = B \cap Q$ and $U
= \alpha(T)$ for short, and let $X = \bigcap_{\phi \in \Aut_\F(Q)} \phi(U)$.

Assume that $S_{T,V}(Q)_\alpha$ is nonzero.  By \cref{P:MP=0} and
\eqref{Wbetaalpha},
\[
B \cap \phi(U) = B_Q
\]
for every choice of $\phi \in \Aut_\F(Q)$. In particular, $B_Q \leq X \leq B \cap U$.

Since $T$ is not $\F$-centric, $U$ and $B_Q$ are not $\F$-centric, and $B_Q < B$. 
Thus $Q < QB$, and so $Q < N_{QB}(Q) = QN_B(Q)$. 
As $B$ is normal in $S$, $N_B(Q)$ is normal in $N_S(Q)$. 
Let $C$ be a normal subgroup of $N_S(Q)$ minimal subject to $B_Q < C \leq N_B(Q)$. 
Since $C$ normalizes $Q$ and $B$ is normal in $S$, we have $[C,Q] \leq B_Q$, and hence
\[
[C,\phi(U)] \leq [C,Q] \leq B_Q \leq C \cap X \leq C \cap \phi(U)
\]
for each $\phi \in \Aut_\F(Q)$. That is, 
\begin{equation}
\label{C and phi(U)}
C \leq N_S(\phi(U)) \quad \text{ and } \quad  \phi(U) \leq N_S(C)
\end{equation}
for each $\phi \in \Aut_\F(Q)$.

Fix a radical $p$-chain $1 = Q_0 < Q_1 < \cdots < Q_m$ of $G$ with $m \geq 1$,
and let $H = N_G(Q_1,\dots,Q_m)$ be its normalizer. Conjugating in $G$ in
order to take $Q_m \leq R$, we will show that the hypotheses of
\cref{L:AKO5.27}\eqref{Lambda:radical} hold for the conjugate chain, and then
conjugating back, it holds for the one just fixed. In this way we are reduced
to $Q_m \leq R$.  

Let $\phi \in \Aut_\F(Q)$ be arbitrary.  
Recall that $C$ was chosen so that $C/B_Q$ is a minimal normal subgroup of $N_S(Q)/B_Q$ (contained in $N_{B}(Q)/B_Q$.
As a minimal normal subgroup of a finite $p$-group, $\ol{C} := QC/Q \cong C/B_Q$ is therefore of order $p$ and contained in the center of $R = \Out_S(Q) =  N_S(Q)/Q$, the last equality because $Q \in \F^c$.
Thus, $\ol{C} \leq H$.  
Since $C \leq B$, we have $[C,\phi(U)] \leq \phi(U) \cap B$ by \eqref{C and phi(U)}, and so $[C,C,\phi(U)] \leq [C,B] = 1$ as $B$ is abelian.
In particular,
$[C,C,V_1] = 0$ for every $\Aut_{\F}(\phi(U))$-composition factor $V_1$ of $\phi(U)$.  
As $\phi\alpha \otimes V$ is an $\Aut_\F(\phi\alpha(U))$-composition factor of $H^j(\phi\alpha(T), k)$,
we have 
\[
\text{$(c-1)^{j+1}(\phi\alpha \otimes V) = 0$ for all $c \in C$} 
\]
by \cref{L:miyamoto} applied with $\Aut_{\F}(\phi(U))$, $\phi(U)$, $k$, $1$, and $2$ in the roles of $G$, $P$, $A$, $n$ and $h$. 
So $(c-1)^{j+1}W_{\phi\alpha} = 0$ for all $c \in C$
as $W_{\phi\alpha}$ is a $k$-submodule of $\phi\alpha \otimes V$. 
Since this holds for all $\phi \in \Aut_\F(Q)$, we have
by the direct sum decomposition \eqref{STVQalpha} that
\[
\text{$(c-1)^{j+1}S_{T,V}(Q)_{\alpha} = 0$ for all $c \in C$}. 
\]
\cref{L:tracebasic} now applies with $H$ in the role of $G$, with $1$ in the role
of $H$, and with $g$ a generator of $\ol{C}$. As $j+1 < p$, we have
$\tr_1^H(S_{T,V}(Q)_{\alpha}) = 0$ by that lemma. Therefore,
\cref{L:AKO5.27}\eqref{Lambda:radical} and \eqref{Lambdadecomp} combine to give
$\lim^m_{\O(\F^c)}({S_{T,V}}^*)_Q \cong \Lambda^m(\Out_\F(Q),S_{T,V}(Q)) = 0$. 
\end{proof}

\begin{proof}[Proof of \cref{STVacyc}]
Let $S_{T,V}$ be a composition factor of $H^j(-,k)$ as a Mackey functor on $\O(\F)$, and suppose that $j \leq p-2$.
If $T$ is centric, then ${S_{T,V}}^*|_{\O(\F^c)}$ is acyclic by \cref{L:DP43}. 
If $T$ is not $\F$-centric, then \cref{P:pruning} shows that the atomic functor $({S_{T,V}}^*)_Q$ is acyclic for each $Q \in \F^c$, so again ${S_{T,V}}^*|_{\O(\F^c)}$ is acyclic by \cref{L:filterbyatomic}.
\end{proof}

\begin{proof}[Proof of \cref{T:main}]
As made explicit in the proof of \cite[Proposition~4.3]{DiazPark2015}, 
given a filtration of $H^j(-,\FF_p)$ whose successive quotients are simple Mackey functors for $\F$,
the restrictions to $\O(\F^c)$ of the contravariant parts of the members of the filtration 
yield a filtration for $H^j(-,\FF_p)|_{\O(\F^c)}$.
So the theorem follows from \cref{STVacyc}. 
\end{proof}

\bibliographystyle{amsalpha}{ }
\bibliography{limh1-oddp.bbl}
\end{document}